\documentclass[11pt]{extarticle}

\usepackage{amssymb,amsfonts,amsthm}
\usepackage{amsmath}
\usepackage{mathrsfs}
\usepackage[top=1in, bottom=1.25in, left=0.75in, right=0.75in]{geometry}
\usepackage{authblk}
\usepackage{hyperref}

\newtheorem{theorem}{Theorem}[section]
\newtheorem{lemma}[theorem]{Lemma}

\newtheorem{conjecture}[theorem]{Conjecture}
\theoremstyle{definition}

\newtheorem{proposition}[theorem]{Proposition}

\theoremstyle{remark}

\numberwithin{equation}{section}

\newcommand{\E}{\mathbb{E}}
\newcommand{\Pp}{\mathbb{P}}
\newcommand{\Z}{\mathbb{Z}}
\newcommand{\TV}{\mathrm{TV}}
\newcommand{\alphabet}{\mathfrak{A}}
\newcommand{\Words}{\mathcal{W}}
\newcommand{\Ent}{\mathsf{H}}
\newcommand{\Info}{\mathsf{I}}
\newcommand{\inP}{\xrightarrow{\ \Pp\ }}

\DeclareMathOperator{\Var}{Var}
\DeclareMathOperator{\Varent}{Varent}
\DeclareMathOperator{\tr}{tr}
\DeclareMathOperator{\Cay}{Cay}

\begin{document}

\title{Cutoff for random Cayley graphs of finite groups}

\author{Haoran Zhu\thanks{Division of Mathematical Sciences, School of Physical and Mathematical Sciences, Nanyang Technological University, 21 Nanyang Link, Singapore 637371. }}%

\date{}

\maketitle

\begin{abstract}
We study the random walk on a Cayley graph of a finite group whose generators are chosen independently and uniformly at random.
We show that, for every finite group, the walk exhibits cutoff with high probability whenever the number of generators is at least a constant multiple of the logarithm of the order of the group and small compared with every positive power of it.
This holds for both undirected and directed Cayley graphs, in discrete and continuous time, and the lower bound on the number of generators is best possible, up to the value of the constant.
The cutoff time is the time at which the expected entropy of the walk comes close to its maximum, and in general it depends on the structure of the group.
This settles a conjecture of Aldous and Diaconis.
\end{abstract}

\section{Introduction}\label{sec:intro}

Many random walks on finite groups \cite{Dia88,SC04} converge to equilibrium abruptly: the total variation distance from the uniform distribution stays close to one for a long time and then drops to near zero within a time interval of smaller order than the mixing time.
This behaviour is known as the \textit{cutoff phenomenon}, and the time interval as the \textit{cutoff window}.
Classical examples include random transpositions, studied by Diaconis and Shahshahani \cite{DS81}, and riffle shuffles, studied by Aldous \cite{Ald83} and by Bayer and Diaconis \cite{BD92}.
The term itself was coined by Aldous and Diaconis \cite{AD86}.
Understanding which Markov chains exhibit cutoff is a central problem in the theory of mixing times; see the survey of Diaconis \cite{Dia96} and \cite[Chapter~18]{LP17}.

In 1985, Aldous and Diaconis \cite[p.~40]{AD85} asked whether cutoff is typical among random walks on finite groups, and suggested testing this by choosing the generators at random.
They conjectured that the resulting walk exhibits cutoff with high probability whenever the number of generators is large compared with the logarithm of the order of the group, but small compared with every positive power of it.
We state the conjecture in the form given in \cite[Section~1.1.1]{HOT26}, which applies to both undirected and directed Cayley graphs, and give precise definitions of the walks and of cutoff later in this section.
The conjecture is also stated in this form for undirected Cayley graphs in \cite[Section~1.1.1]{HH24} and \cite[Conjecture~1]{HR25}.

\begin{conjecture}\label{conj:AD}
Let $G$ range over a sequence of finite groups of order $n\to\infty$, and let $k=k(G)$ satisfy $k/\log n\to\infty$ and $\log k/\log n\to0$.
Then the random walks on the Cayley graph and on the directed Cayley graph of $G$ generated by $k$ elements chosen independently and uniformly at random exhibit cutoff with high probability.
\end{conjecture}

Aldous and Diaconis also suggested that the cutoff time should depend only on $n$ and $k$.
This turned out to be false: Hermon and Olesker-Taylor \cite[Theorem~A]{HOT19} showed that, for groups of unit upper triangular matrices, the cutoff time also depends on the order of the abelianisation.
In this paper, we prove Conjecture~\ref{conj:AD} for all finite groups, and we identify the cutoff time in terms of the entropy of the walk (Theorem~\ref{thm:main}).
In fact, we show that the conclusion of the conjecture holds whenever $k\geqslant2^{14}\log n$ and $\log k/\log n\to0$, in discrete and continuous time.

Conjecture~\ref{conj:AD} has stimulated a large body of work on random walks with random generators, known as random random walks \cite{DH96,Gre97,H94,Pak99,R96,Wil97}; see the survey of Hildebrand \cite{H05}.
Dou and Hildebrand \cite[Theorem~1]{DH96} proved an upper bound on the mixing time that is valid for every group, and Roichman \cite[Theorems~1 and~2]{R96} later gave a simpler proof.
Hildebrand \cite[Theorem~3]{H94} proved a matching lower bound for abelian groups, and Dou and Hildebrand \cite[Theorem~4]{DH96} extended it to some families of groups whose irreducible representations have bounded degree.
Together, these results give cutoff at time $\log n/\log(k/\log n)$ for abelian groups and for these families of groups when $k\geqslant(\log n)^{1+\delta}$ for some constant $\delta>0$.
When $\log k/\log\log n\to\infty$, this upper bound agrees to leading order with the lower bound given by the diameter, so the conjecture holds for every group in this range \cite[Section~1.3.1]{HOT26}.
In a related direction, Alon and Roichman \cite{AR94} showed that a Cayley graph generated by a logarithmic number of random elements is an expander with high probability.

More recently, Hermon and Olesker-Taylor related the mixing of random walks with random generators to the entropy of an auxiliary random walk on $\Z^k$, which records how many times each generator has been applied, counted with sign.
Using this approach, they established cutoff for all abelian groups when $k\to\infty$, $k=O(\log n)$, and $k-d(G)\to\infty$, where $d(G)$ is the minimal size of a generating set of $G$ \cite[Theorem~A]{HOT26}.
They also treated groups of unit upper triangular matrices and Heisenberg groups \cite{HOT19}, and Hermon and Huang \cite[Theorem~1]{HH24} extended these results to nilpotent groups of bounded rank and nilpotency class.
Huang and Rao \cite[Theorems~1.3 and~1.4]{HR25} treated dihedral groups of order twice a prime and, when $k/\log n\to\infty$, virtually abelian groups with an abelian normal subgroup of bounded index.
All of these groups are close to abelian, and for an abelian group the walk at time $t$ is a function of the generators and of the auxiliary walk at time $t$.
For a general group, however, the walk also depends on the order in which the generators are applied, and Conjecture~\ref{conj:AD} has remained open whenever $k$ is polylogarithmic in $n$.

Entropy also plays a central role in the study of cutoff for random walks on random graphs and on Ramanujan graphs, and for related sparse Markov chains \cite{BHS17,BLPS18,BCS19,BL22,LP16,LS10}.
In each case, cutoff has been established at an entropic time, determined by the growth of the entropy of the walk.
For a random walk on a group of order $n$, the heuristic is as follows.
As long as the information content of the walk, that is, minus the logarithm of the probability of its current position, is typically smaller than $\log n$ by a large margin, the walk is confined to a small part of the group and is far from being mixed.
Once the information content is typically close to $\log n$, the walk is spread fairly evenly over the group, and a spectral gap ensures that the walk is mixed shortly afterwards.
Thus, if the information content is concentrated around its mean, which is the entropy, and the spectral gap is not too small, the walk should exhibit cutoff.
Salez \cite[Theorem~5]{S24} made this precise under the name of entropic concentration, using the \textit{varentropy}, which is the variance of the information content.
He showed that a sequence of Markov chains exhibits cutoff whenever the relaxation time, multiplied by one plus the square root of the varentropy, is of smaller order than the mixing time.
Using this criterion, he established cutoff for random Cayley graphs of abelian groups when $k\geqslant(1+\delta)\log_2n$ for some constant $\delta>0$ and $\log k=o\bigl((\log n)^{1/3}\bigr)$ \cite[Corollary~4]{S24}.

Hermon, Huang, Pedrotti, and Salez \cite[Theorem~1 and Corollary~1]{HHPS26} showed that, for a conjugacy-invariant random walk on a group, the varentropy is at most an absolute constant times the number of generators, uniformly in time and in the size of the group.
A random walk on a random Cayley graph of a non-abelian group is, however, not conjugacy-invariant in general, so this bound does not apply, and bounding the varentropy becomes the main difficulty.

Our main result, Theorem~\ref{thm:main}, shows that cutoff occurs at the entropic time $t_*$ defined in \eqref{eq:centre}, the first time at which the expected entropy of the discrete-time walk comes within $2(\log n)^{2/3}$ of its maximum value $\log n$.
Loosely speaking, the walk mixes as soon as it is, in terms of entropy, almost as spread out as the uniform distribution.
Our second main result, Theorem~\ref{thm:information}, bounds the variance of the information content of the walk, under the joint law of the generators and the walk, uniformly over all finite groups.
Using Theorem~\ref{thm:information}, we also bound the cutoff window (Theorem~\ref{thm:window}).

To make the setting precise, let $G$ be a finite group of order $n$ with identity $\mathrm{id}$, and let $k$ be a positive integer.
We write $[k]=\{1,\dots,k\}$, and all logarithms are natural.

Let $\alphabet=[k]\times\{-1,1\}$.
For $z=(z_1,\dots,z_k)\in G^k$, the \textbf{Cayley graph} $\Cay(G,z)$ is the multigraph with vertex set $G$ in which each vertex $x$ is joined to $xz_i^{\varepsilon}$ for every $(i,\varepsilon)\in\alphabet$.
Parallel edges and loops are allowed, so $\Cay(G,z)$ is regular of degree $2k$; see also~\cite[Definition~1]{HH24} and \cite[Remark~1.2]{HR25}.
At each step, the simple random walk on $\Cay(G,z)$ moves from its current position $x$ to $xz_i^{\varepsilon}$, where $(i,\varepsilon)$ is chosen uniformly at random from $\alphabet$.
Its transition matrix $P_{z}$ is symmetric, so the walk is reversible with respect to the uniform distribution $\pi$ on $G$, and $\pi$ is stationary.
The \textbf{directed Cayley graph} has an edge from each vertex $x$ to $xz_i$ for every $i\in[k]$, and the simple random walk on it is defined in the same way, with $\alphabet$ replaced by $[k]\times\{1\}$.
We call these two walks the \textit{undirected walk} and the \textit{directed walk}.
The transition matrix $P_{z}$ of the directed walk need not be symmetric, but $\sum_{x\in G}P_{z}(x,y)=1$ for every $y\in G$, so $\pi$ is stationary for both walks.

It is convenient to describe the walk by words: for each integer $m\geqslant0$, let $\Words_m=\alphabet^m$ be the set of \textbf{words} of length $m$.
The value at $z$ of a word $w=(w_1,\dots,w_m)$ with letters $w_r=(i_r,\varepsilon_r)$ is
\[
w(z)=z_{i_1}^{\varepsilon_1}z_{i_2}^{\varepsilon_2}\cdots z_{i_m}^{\varepsilon_m},
\]
and the value of the empty word is $\mathrm{id}$.
If a word is chosen uniformly at random from $\Words_m$, then its letters are independent and uniformly distributed on $\alphabet$, so its value at $z$ has law $q_{z,m}=P_{z}^m(\mathrm{id},\cdot)$, the distribution after $m$ steps of the walk started at $\mathrm{id}$.
We also consider the continuous-time walk, which jumps at rate $1$, and its distribution at time $t\geqslant0$, started at $\mathrm{id}$, is
\begin{equation}\label{eq:poisson-law}
p_{z,t}=\sum_{m=0}^{\infty}e^{-t}\frac{t^m}{m!}\,q_{z,m}.
\end{equation}

The \textbf{total variation distance} between probability measures $\mu$ and $\nu$ on $G$ is
\[
\|\mu-\nu\|_{\TV}=\frac12\sum_{x\in G}|\mu(x)-\nu(x)|.
\]
Our results hold for both walks, in discrete and continuous time, and we use the same notation in all cases: we write $\mu_{z,t}$ for the distribution at time $t\geqslant0$ of the walk under consideration, and $d_{z}(t)=\|\mu_{z,t}-\pi\|_{\TV}$.
Thus $\mu_{z,t}=q_{z,\lfloor t\rfloor}$ in discrete time, and $\mu_{z,t}=p_{z,t}$ in continuous time.
Left multiplication by any element of $G$ is an automorphism of both graphs, so this distance does not depend on the starting vertex.
For $\varepsilon\in(0,1)$, the \textbf{mixing time} is $t_{\mathrm{mix}}^{z}(\varepsilon)=\inf\{t\geqslant0:d_{z}(t)\leqslant\varepsilon\}$.

Let $Z=(Z_1,\dots,Z_k)$, where $Z_1,\dots,Z_k$ are chosen independently and uniformly at random from $G$.
We consider sequences of finite groups with $n\to\infty$, where $k=k(G)$ may depend on $G$.
An event holds with \textbf{high probability} if its probability tends to one, and $\inP$ denotes convergence in probability.
Following \cite[Section~1.2.1]{HOT26}, we say that the walk \textit{exhibits cutoff with high probability at time} $\tau$, where $\tau=\tau(G)$ is a positive real number, if
\[
d_{Z}\bigl((1-\delta)\tau\bigr)\inP1\qquad\text{and}\qquad d_{Z}\bigl((1+\delta)\tau\bigr)\inP0
\]
for every fixed $\delta\in(0,1)$, and that it \textit{exhibits cutoff with high probability} if this holds for some choice of $\tau$.
Since $d_{Z}$ is non-increasing, cutoff with high probability at time $\tau$ is equivalent to $t_{\mathrm{mix}}^{Z}(\varepsilon)/\tau\inP1$ for every fixed $\varepsilon\in(0,1)$.
Throughout, $C$ denotes a positive absolute constant whose value may change from one occurrence to the next.
For a random variable $Y$ and an event $E$, we write $\E[Y;\,E]=\E[Y\mathbf 1_E]$, where $\mathbf 1_E$ is the indicator function of $E$.
For positive quantities $a$ and $b$ depending on $G$, we write $a\sim b$ if $a/b\to1$.

For a probability measure $\mu$ on a finite set, we write $\Ent(\mu)=-\sum_x\mu(x)\log\mu(x)$ for its \textbf{entropy}, with the convention $0\log0=0$.
The \textbf{expected entropy} of the walk after $m$ steps is $h_m=\E\,\Ent(q_{Z,m})$, and the \textbf{entropic time} is
\begin{equation}\label{eq:centre}
t_*=\min\bigl\{m\in\Z_{\geqslant0}:h_m\geqslant\log n-2(\log n)^{2/3}\bigr\},
\end{equation}
with $t_*=\infty$ if no such $m$ exists.
Note that $t_*$ is deterministic: it depends only on $G$ and $k$.
Our main result is the following.

\begin{theorem}\label{thm:main}
Let $G$ range over a sequence of finite groups of order $n\to\infty$, and let $k=k(G)$ satisfy
\begin{equation}\label{eq:regime}
k\geqslant2^{14}\log n\qquad\text{and}\qquad\frac{\log k}{\log n}\to0.
\end{equation}
For each $G$, let $Z=(Z_1,\dots,Z_k)$, where $Z_1,\dots,Z_k$ are chosen independently and uniformly at random from $G$.
Then, for both the undirected and the directed walks, in discrete and continuous time, $t_*$ is finite for all sufficiently large $n$, and the walk exhibits cutoff with high probability at time $t_*$.
\end{theorem}

In particular, Conjecture~\ref{conj:AD} holds for every sequence of finite groups, and $t_{\mathrm{mix}}^{Z}(\varepsilon)/t_*\inP1$ for every fixed $\varepsilon\in(0,1)$.
In the range \eqref{eq:regime}, Theorem~\ref{thm:main} extends the cutoff results described above for abelian, nilpotent, and virtually abelian groups to all finite groups.

The first condition in \eqref{eq:regime} is best possible, up to the value of the constant $2^{14}$, and the second is necessary for cutoff in continuous time.
Indeed, if $G=(\Z/2\Z)^d$ and $k<d$, then $Z$ generates a subgroup of order at most $2^k<n$, to which the walk is confined.
If instead $\log k/\log n$ is bounded away from zero, then the mixing time is bounded with high probability \cite[Section~7.2]{HOT26}, and the continuous-time walk cannot exhibit cutoff, since $p_{z,t}(\mathrm{id})\geqslant e^{-t}$ for every $z\in G^k$ and every $t\geqslant0$.

Wilson \cite[Theorem~1]{Wil97} established cutoff for $(\Z/2\Z)^d$.
In the form given by Hermon and Olesker-Taylor \cite[Section~1.1.1]{HOT26}, his conjecture \cite[Conjecture~7]{Wil97} asserts that, when $k-\log_2n\to\infty$ and $\log k/\log n\to0$, this group gives rise to the slowest mixing time among all groups of order at most $2^d$.
Hermon and Olesker-Taylor verified this conjecture for nilpotent groups \cite[Theorem~C]{HOT26}, and in work in preparation \cite{Zhu26}, we prove it for all finite groups, for both undirected and directed Cayley graphs.

When $k/\log n\to\infty$, by Lemma~\ref{lem:centre-bounds}, the entropic time lies between two classical bounds:
\[
(1-o(1))\frac{\log n}{\log k}\leqslant t_*\leqslant(1+o(1))\frac{\log n}{\log(k/\log n)}.
\]
The lower bound comes from the diameter, and the upper bound is the cutoff time for abelian groups.
Since cutoff times are unique up to a factor $1+o(1)$, Theorem~\ref{thm:main} and \cite[Section~1.3.1]{HOT26} show that $t_*$ is asymptotic to the upper bound for abelian groups.
On the other hand, by \cite[Theorem~A]{HOT19}, it is asymptotic to the lower bound for the groups of $d\times d$ unit upper triangular matrices over $\Z/p\Z$ when $k$ exceeds a suitable power of $\log n$, where $p$ is prime and $d\geqslant3$ is bounded.

The two bounds have a simple interpretation in terms of entropy.
Each step increases the entropy of the walk by at most $\log(2k)$, with near equality as long as distinct words typically take distinct values.
Hence the entropy cannot reach $\log n$ in fewer than $\log n/\log(2k)$ steps, and this number is asymptotic to the lower bound.
For an abelian group, the walk is determined by the auxiliary walk on $\Z^k$, whose entropy after $m$ steps is about $m\log(k/m)$ when $m$ is small compared with $k$.
This entropy reaches $\log n$ when $m\log(k/m)=\log n$, and the solution of this equation is asymptotic to the upper bound.
In general, where $t_*$ lies between the two bounds depends on how often distinct words take the same value, that is, on the structure of $G$.

Next, we state our second main result, from which we deduce Theorem~\ref{thm:main}.
If $X$ is a random variable with law $\mu$ on a finite set, then $-\log\mu(X)$ is the \textbf{information content} of $X$.
Its mean is $\Ent(\mu)$, and its variance is the varentropy $\Varent(\mu)$ of $\mu$ \cite{HHPS26,S24}.
For an integer $m\geqslant0$, let $W$ be chosen uniformly at random from $\Words_m$, independently of $Z$, and let
\[
\Info_m=-\log q_{Z,m}\bigl(W(Z)\bigr).
\]
Conditional on $Z$, the random variable $\Info_m$ is the information content of the walk after $m$ steps.
We write $\Var(\Info_m)$ for the variance of $\Info_m$ under the joint law of $Z$ and $W$.

\begin{theorem}\label{thm:information}
Let $G$ be a finite group, and let $k$ and $m$ be integers with $1\leqslant m\leqslant k/4096$.
Then
\[
\Var(\Info_m)\leqslant Cm\log^2(2k),
\]
where $C$ is an absolute constant.
\end{theorem}

Since $\E[\Info_m\mid Z]=\Ent(q_{Z,m})$, by the law of total variance we have
\[
\Var(\Info_m)=\E\bigl[\Varent(q_{Z,m})\bigr]+\Var\bigl(\Ent(q_{Z,m})\bigr).
\]
Thus Theorem~\ref{thm:information} bounds both the varentropy of the walk for a typical choice of generators and the fluctuations of its entropy from one choice of generators to another.
The first is what entropic concentration requires, and the second is what allows the cutoff time to be deterministic.
Salez \cite[Theorem~6]{S24} proved a bound of the same order $t\log^2(2k)$ on the varentropy at the mixing time $t$ for random Cayley graphs of abelian groups, which are non-negatively curved.
In contrast to this bound and to the estimate of Hermon, Huang, Pedrotti, and Salez \cite[Theorem~1]{HHPS26}, Theorem~\ref{thm:information} requires neither non-negative curvature nor invariance under conjugation.

By Theorem~\ref{thm:information}, the information content of the walk after $m$ steps is concentrated around $h_m$ with a standard deviation of order at most $\sqrt m\,\log(2k)$.
By the theorem of Alon and Roichman (Lemma~\ref{lem:spectrum}), with high probability, the $L_2$ distance of the walk from $\pi$ is multiplied by at most $(e^3\lceil\log n\rceil/k)^{1/2}$ at each step.
Hence, once the information content typically exceeds $\log n-3\omega$, where $\sqrt m\,\log(2k)=o(\omega)$, a further $O\bigl(\omega/\log(k/\log n)\bigr)$ steps suffice for the walk to mix.
This leads to the following bound on the cutoff window.
For $\omega>0$, let
\[
t_\omega=\min\bigl\{m\in\Z_{\geqslant0}:h_m\geqslant\log n-2\omega\bigr\},
\]
so that $t_*=t_{(\log n)^{2/3}}$.

\begin{theorem}\label{thm:window}
Let $G$ range over a sequence of finite groups of order $n\to\infty$, and let $k=k(G)\geqslant2^{14}\log n$.
For each $G$, let $Z=(Z_1,\dots,Z_k)$, where $Z_1,\dots,Z_k$ are chosen independently and uniformly at random from $G$, and let $\omega=\omega(G)>0$ satisfy
\[
\sqrt{t_\omega}\,\log(2k)=o(\omega)\qquad\text{and}\qquad\omega=o(t_\omega).
\]
Then, for both the undirected and the directed walks, in discrete and continuous time,
\[
\E\,d_{Z}(t_\omega-1)\to1\qquad\text{and}\qquad\E\,d_{Z}\Bigl(t_\omega+\frac{20\,\omega}{\log(k/\log n)}\Bigr)\to0.
\]
\end{theorem}

By Lemma~\ref{lem:omega-bounds}, we have $t_\omega\leqslant\lceil\log n\rceil$, and hence the first condition in Theorem~\ref{thm:window} holds whenever $\sqrt{\log n}\,\log(2k)=o(\omega)$.
For example, if $k=O(\log n)$, then both conditions hold whenever $\sqrt{\log n}\,\log\log n=o(\omega)$ and $\omega\log\log n=o(\log n)$, and then the cutoff window is of order at most $\omega$.
For comparison, Hermon and Olesker-Taylor \cite[Theorem~2.4 and Remark~2.5]{HOT19b} showed that, under some conditions on an abelian group, the window of the continuous-time walk is of order $\sqrt{\log n}$ when $k$ is of order $\log n$.
When $\log k\leqslant(\log\log n)^2$, we deduce Theorem~\ref{thm:main} from Theorem~\ref{thm:window} with $\omega=(\log n)^{2/3}$, and when $\log k>(\log\log n)^2$, the two classical bounds agree to leading order and already give cutoff.

The main idea in the proof of Theorem~\ref{thm:information} is a change of coordinates.
We use the Efron--Stein inequality \cite{ES81,Ste86}, which requires an estimate of how much the information content changes when a single generator or letter is replaced with an independent copy.
In a non-abelian group, this change is hard to estimate directly: replacing a generator moves the value of a word by an element that depends on the whole word, so different words are moved differently.
When $m$ is small compared with $k$, as it is for $m$ of order $t_*$, most letters of a word chosen uniformly at random carry an index that occurs nowhere else in the word.
The generators of these indices are independent and uniformly distributed, and each of them is used only once, so the positions of the walk just after these letters are also independent and uniformly distributed.
We use these positions, rather than the generators, as coordinates for the Efron--Stein inequality.
In the new coordinates, replacing a single coordinate changes the products of letters only on a few short intervals, each of which ends at such a position or at the end of the word.
Only a change on an interval ending at the end of the word can move the walk, and then it moves the walk by a right translation.
After conditioning on the letters in these intervals and on the occurrences of a few indices, the walk and its conditional law are moved by the same right translation, which does not change the information content.
The conditioning changes the information content by at most of order $\log(2k)$ times the number of letters and occurrences involved, and the second moment of this number is bounded, which explains the factor $\log^2(2k)$ in Theorem~\ref{thm:information}.

The paper is organised as follows.
In Section~\ref{sec:coordinates}, we construct the new coordinates and show that, in a typical word, the indices occurring only once are rarely far apart.
In Section~\ref{sec:change}, we bound the effect of replacing a single coordinate, and in Section~\ref{sec:variance} we combine these bounds with the Efron--Stein inequality to prove Theorem~\ref{thm:information}.
Finally, in Section~\ref{sec:cutoff}, we deduce Theorems~\ref{thm:main} and~\ref{thm:window} from Theorem~\ref{thm:information} and the theorem of Alon and Roichman.

\section{Words and independent coordinates}\label{sec:coordinates}

In this section, we construct the coordinates described in the introduction, and we show that a typical word is well behaved.
The coordinates are designed for the Efron--Stein inequality, which bounds the variance of a function of independent random variables by the effect of replacing each of them with an independent copy.
A single generator affects the walk at every occurrence of its index in the word, whereas a single coordinate will affect the products of letters only on a few short intervals.

Let $w=(w_1,\dots,w_m)\in\Words_m$ with letters $w_r=(i_r,\varepsilon_r)$.
For $i\in[k]$, the \textbf{multiplicity} $N_i(w)$ of $i$ in $w$ is the number of positions $r$ with $i_r=i$.
The \textit{occurrences} of $i$ in $w$ are the pairs $(r,\varepsilon_r)$ with $i_r=i$.
An index $i$ is a \textbf{singleton} of $w$ if $N_i(w)=1$.
A position $r$ is a \textbf{singleton position} of $w$ if $i_r$ is a singleton of $w$, and we write $S(w)$ for the set of singleton positions of $w$.
An \textit{interval} is a set $(a,c]=\{a+1,\dots,c\}$ of consecutive positions, where $0\leqslant a\leqslant c\leqslant m$.
The \textit{restriction} of $w$ to an interval $I$ is the word $w_I$ formed by the letters of $w$ at the positions in $I$.
For example, $w_{(a,c]}=(w_{a+1},\dots,w_c)$, and $w_{(a,a]}$ is the empty word.

Let $B=(B_1,\dots,B_k)$ and $U=(U_1,\dots,U_m)$, where $B_1,\dots,B_k,U_1,\dots,U_m$ are chosen independently and uniformly at random from $G$, and let $U_0=\mathrm{id}$.
Let $w\in\Words_m$ with $S(w)=\{j_1<\dots<j_s\}$, and let $j_0=0$.
We define $\Psi_w(B,U)$ to be the tuple $z\in G^k$ such that $z_i=B_i$ whenever $i$ is not a singleton of $w$, and
\begin{equation}\label{eq:prefix}
w_{(0,j_\ell]}(z)=U_{j_\ell}\qquad\text{for }1\leqslant\ell\leqslant s.
\end{equation}
To see that this tuple exists and is unique, note that the indices at the positions strictly between $j_{\ell-1}$ and $j_\ell$ are not singletons, so their generators are given by $B$.
Hence \eqref{eq:prefix} determines the generators at the singleton positions one at a time: if $1\leqslant\ell\leqslant s$ and $w_{j_\ell}=(i,\varepsilon)$, then
\begin{equation}\label{eq:prefix-coordinates}
z_i^{\varepsilon}=\bigl(U_{j_{\ell-1}}\,w_{(j_{\ell-1},\,j_\ell-1]}(B)\bigr)^{-1}U_{j_\ell}.
\end{equation}
In other words, $\Psi_w$ takes from $B$ the generators of the indices that are not singletons, and chooses each remaining generator so that the product of the letters up to its position is prescribed.
With this choice, replacing a single coordinate has a local effect.
For instance, by \eqref{eq:prefix-coordinates}, the coordinate $U_{j_\ell}$ enters only the generators at the singleton positions $j_\ell$ and, when $\ell<s$, $j_{\ell+1}$.
The coordinates $U_r$ with $r\notin S(w)$, and $B_i$ with $i$ a singleton of $w$, are not used.
If $S(w)=\varnothing$, then $\Psi_w(B,U)=B$.

Let $W$ be chosen uniformly at random from $\Words_m$, independently of $(B,U)$.
The coordinates of $(B,U,W)$ are $B_1,\dots,B_k$, $U_1,\dots,U_m$, and the letters $W_1,\dots,W_m$ of $W$, and they are independent.
The first lemma shows that these coordinates produce a uniformly distributed tuple of generators, independent of the word, even after a single coordinate has been replaced with an independent copy.

\begin{lemma}\label{lem:prefix-law}
Let $(B',U',W')$ be obtained from $(B,U,W)$ by replacing one coordinate with an independent copy.
Then each of $\Psi_W(B,U)$ and $\Psi_{W'}(B',U')$ is uniformly distributed on $G^k$ and independent of $(W,W')$.
\end{lemma}

\begin{proof}
We first show that $\Psi_w(B,U)$ is uniformly distributed on $G^k$ for every $w\in\Words_m$.
Let $z=\Psi_w(B,U)$ and $S(w)=\{j_1<\dots<j_s\}$.
The coordinates of $z$ at the indices that are not singletons of $w$ are the corresponding coordinates of $B$, which are independent and uniformly distributed, and we condition on them.
By \eqref{eq:prefix-coordinates}, the generators at the singletons of $w$ are then determined in turn by $U_{j_1},\dots,U_{j_s}$, and conversely $U_{j_\ell}=w_{(0,j_\ell]}(z)$ for each $\ell$, by \eqref{eq:prefix}.
This gives a bijection between $(U_{j_1},\dots,U_{j_s})$ and the generators at the singletons of $w$.
Given the remaining generators, the generators at the singletons are therefore independent and uniformly distributed, and hence $z$ is uniformly distributed on $G^k$.

Now each of $(B,U)$ and $(B',U')$ has the law of $(B,U)$ and is independent of $(W,W')$.
Hence, conditional on $(W,W')=(w,w')$, the tuples $\Psi_W(B,U)$ and $\Psi_{W'}(B',U')$ have the laws of $\Psi_w(B,U)$ and $\Psi_{w'}(B,U)$, respectively, and both laws are uniform on $G^k$.
These conditional laws do not depend on $(w,w')$, and the lemma follows.
\end{proof}

The number and lengths of the intervals affected by a replacement depend on the multiplicities of the indices and on the gaps between singleton positions.
The \textit{blocks} of a set $R$ of positions are the intervals $(r,r']$ such that $r<r'$ are consecutive elements of $R\cup\{0,m\}$.
For example, if $S(w)=\{j_1<\dots<j_s\}$ and $j_0=0$, then the blocks of $S(w)$ are $(j_{\ell-1},j_\ell]$ for $1\leqslant\ell\leqslant s$, together with $(j_s,m]$ if $j_s<m$.
The next lemma shows that, in a typical word, the blocks containing a few given positions are short, even when the letters at these positions are prescribed.

\begin{lemma}\label{lem:blocks}
Let $k$ and $m$ be integers with $1\leqslant m\leqslant k/4096$, let $T\subseteq[m]$ be non-empty, and let $K\subseteq[k]$ with $|K|\geqslant k/2$.
Let $W$ be a random element of $\Words_m$ whose letters at the positions in $T$ are prescribed, and whose letters at the remaining positions are independent, with indices uniformly distributed on $K$.
Let $L$ be the sum of $|I|$ over the blocks $I$ of $S(W)\setminus T$ that contain a position of $T$.
Then $\E[L^2]\leqslant C|T|^4$.
\end{lemma}

\begin{proof}
We first bound the probability that a set of positions outside $T$ contains no singleton position.
Let $J\subseteq[m]\setminus T$ be a set of $b\geqslant1$ positions, and condition on the letters of $W$ at the positions outside $J$, among which at most $m$ distinct indices occur.
If no position of $J$ is a singleton position of $W$, then each position of $J$ carries either an index that occurs outside $J$, or an index that occurs only in $J$ and at least twice there.
Hence at least one of the following holds:
\begin{enumerate}
\item[(i)] at least $b/2$ positions of $J$ carry indices occurring outside $J$;
\item[(ii)] more than $b/2$ positions of $J$ carry indices that occur only in $J$, each at least twice.
\end{enumerate}

In case~(i), there are at most $2^b$ choices for the set of such positions.
Given the letters outside $J$, each of these positions independently carries one of at most $m$ prescribed indices with probability at most $m/|K|$, so case~(i) has probability at most $2^b(m/|K|)^{b/2}$.
In case~(ii), each of the indices in question occurs at least twice in $J$, so, reading from left to right, at least $b/4$ positions of $J$ carry an index that has already occurred earlier in $J$.
There are again at most $2^b$ choices for the set of these positions.
Given the earlier letters, each of them carries one of the at most $b$ indices already seen in $J$ with probability at most $b/|K|$, so case~(ii) has probability at most $2^b(b/|K|)^{b/4}$.
Since $b\leqslant m\leqslant k/4096$ and $|K|\geqslant k/2$, we have $m/|K|\leqslant1/2048$ and $b/|K|\leqslant1/2048$, and hence
\begin{equation}\label{eq:no-singleton}
\Pp\bigl(J\cap S(W)=\varnothing\bigr)\leqslant2^b\,2048^{-b/2}+2^b\,2048^{-b/4}\leqslant2\cdot3^{-b}.
\end{equation}

Now fix $j\in T$, let $I_j$ be the block of $S(W)\setminus T$ containing $j$, and let $L_j=|I_j\setminus(T\cup S(W))|$.
Every position of $I_j$ other than its right endpoint lies outside $S(W)\setminus T$, so $|I_j|\leqslant|T|+L_j+1$.
Since $I_j$ is an interval containing $j$, the positions counted by $L_j$ that are smaller than $j$ are the largest positions $r<j$ with $r\notin T$, and those that are larger than $j$ are the smallest positions $r>j$ with $r\notin T$.
Hence, if $L_j\geqslant\ell$, then the $\lceil\ell/2\rceil$ largest positions $r<j$ with $r\notin T$, or the $\lceil\ell/2\rceil$ smallest positions $r>j$ with $r\notin T$, are not singleton positions of $W$.
By \eqref{eq:no-singleton}, it follows that $\Pp(L_j\geqslant\ell)\leqslant4\cdot3^{-\ell/2}$, and hence
\[
\E\bigl[L_j^2\bigr]=\sum_{\ell\geqslant1}(2\ell-1)\,\Pp(L_j\geqslant\ell)\leqslant\sum_{\ell\geqslant1}4(2\ell-1)3^{-\ell/2}\leqslant C.
\]
The blocks of $S(W)\setminus T$ that contain a position of $T$ are among the intervals $I_j$, so $L\leqslant\sum_{j\in T}|I_j|\leqslant|T|(|T|+1)+\sum_{j\in T}L_j$.
By the Cauchy--Schwarz inequality, it follows that
\[
L^2\leqslant2|T|^2(|T|+1)^2+2|T|\sum_{j\in T}L_j^2,
\]
and taking expectations, we obtain $\E[L^2]\leqslant C|T|^4$, as required.
\end{proof}

\section{Changing a single coordinate}\label{sec:change}

In this section, we compare the information content of the walk before and after a single coordinate of $(B,U,W)$ is replaced with an independent copy.
We make the comparison on the fibres of a suitable function of the pair of words, chosen so that on each fibre the conditional laws of the walk before and after the replacement differ by a right translation.

We begin with two estimates for the logarithm of a density.
If $E$ is an event with $\Pp(E)=\alpha\in(0,1)$, then the law of the walk at a given time is $\alpha\nu+(1-\alpha)\nu'$, where $\nu$ and $\nu'$ are its conditional laws given $E$ and given the complement of $E$.
The following lemma compares this mixture with $\nu$.
Part~(i) is a variant of an estimate of Eberhard and Varj\'u \cite[Lemma~2.2]{EV21}; see also \cite[proof of Proposition~3.14]{HR25}.

\begin{lemma}\label{lem:mixture}
Let $\mu$ and $\nu$ be probability measures on a finite set, let $X$ be a random variable with law $\nu$, and let $\alpha\in(0,1]$.
\begin{enumerate}
\item[(i)] If $\mu\geqslant\alpha\nu$, then $\E\bigl[\log^2\bigl(\mu(X)/\nu(X)\bigr)\bigr]\leqslant\log^2(1/\alpha)+2$.
\item[(ii)] If $\alpha\geqslant1/2$ and $\mu=\alpha\nu+(1-\alpha)\nu'$ for some probability measure $\nu'$, then $\E\bigl[\log^2\bigl(\mu(X)/\nu(X)\bigr)\bigr]\leqslant3(1-\alpha)$.
\end{enumerate}
\end{lemma}

\begin{proof}
Let $Y=\log\bigl(\mu(X)/\nu(X)\bigr)$, and let $Y_+$ and $Y_-$ be its positive and negative parts, so that $Y^2=Y_+^2+Y_-^2$.
Since $\E\bigl[\mu(X)/\nu(X)\bigr]=\sum_{x:\nu(x)>0}\mu(x)\leqslant1$, by Markov's inequality we have $\Pp(Y>s)\leqslant e^{-s}$ for $s\geqslant0$, and hence
\[
\E\bigl[Y_+^2\bigr]=\int_0^\infty2s\,\Pp(Y>s)\,\mathrm{d}s\leqslant\int_0^\infty2se^{-s}\,\mathrm{d}s=2.
\]
In case~(i), we have $Y_-\leqslant\log(1/\alpha)$ by hypothesis, and (i) follows.

In case~(ii), we have $Y\geqslant\log\alpha\geqslant-2(1-\alpha)$ since $\alpha\geqslant1/2$, and hence $\E\bigl[Y_-^2\bigr]\leqslant4(1-\alpha)^2\leqslant2(1-\alpha)$.
Using the inequality $\log^2u\leqslant u-1$ for $u\geqslant1$, we also have
\[
\E\bigl[Y_+^2\bigr]\leqslant\sum_{x:\nu(x)>0}\nu(x)\Bigl(\frac{\mu(x)}{\nu(x)}-1\Bigr)_+\leqslant\sum_x\bigl(\mu(x)-\nu(x)\bigr)_+=(1-\alpha)\sum_x\bigl(\nu'(x)-\nu(x)\bigr)_+\leqslant1-\alpha.
\]
Adding the two bounds proves~(ii).
\end{proof}

The next lemma allows us to sum such estimates over the values of a random variable.

\begin{lemma}\label{lem:entropy-sum}
Let $A$ be a random variable taking finitely many values, let $\mathcal V$ be a non-empty set of values of $A$, each of positive probability, and let $p=\Pp(A\in\mathcal V)$.
Let $c$ be a non-negative function on $\mathcal V$ such that $\sum_{a\in\mathcal V}e^{-c(a)}\leqslant1$.
Then
\[
\sum_{a\in\mathcal V}\Pp(A=a)\log^2\frac{1}{\Pp(A=a)}\leqslant8p\Bigl(1+\log^2\frac1p\Bigr)+4\,\E\bigl[c(A)^2;\,A\in\mathcal V\bigr].
\]
\end{lemma}

\begin{proof}
For $a\in\mathcal V$, let $\theta_a=\Pp(A=a)/p$, so that $\sum_{a\in\mathcal V}\theta_a=1$.
For $u\geqslant0$, the values $a\in\mathcal V$ with $\log(1/\theta_a)>c(a)+u$ satisfy $\theta_a<e^{-u}e^{-c(a)}$, so the sum of $\theta_a$ over these values is less than $e^{-u}$.
Hence
\[
\sum_{a\in\mathcal V}\theta_a\Bigl(\log\frac1{\theta_a}-c(a)\Bigr)_+^2=\int_0^\infty2u\sum_{a:\log(1/\theta_a)>c(a)+u}\theta_a\,\mathrm{d}u\leqslant\int_0^\infty2ue^{-u}\,\mathrm{d}u=2,
\]
and therefore $\sum_a\theta_a\log^2(1/\theta_a)\leqslant2\sum_a\theta_ac(a)^2+4$.
Since $\log(1/\Pp(A=a))=\log(1/p)+\log(1/\theta_a)$, using the inequality $(x+y)^2\leqslant2x^2+2y^2$, we obtain
\[
\sum_{a\in\mathcal V}\Pp(A=a)\log^2\frac1{\Pp(A=a)}\leqslant2p\log^2\frac1p+4\sum_{a\in\mathcal V}\Pp(A=a)c(a)^2+8p,
\]
and the lemma follows.
\end{proof}

We now fix the setting in which these estimates will be applied.
Let $(Z,W,Z',W')$ be random variables such that $W$ and $W'$ are uniformly distributed on $\Words_m$.
Suppose that each of $Z$ and $Z'$ is uniformly distributed on $G^k$ and independent of $(W,W')$.
Let $F=-\log q_{Z,m}\bigl(W(Z)\bigr)$ and $F'=-\log q_{Z',m}\bigl(W'(Z')\bigr)$.
Let $A$ be a function of $(W,W')$ taking finitely many values.
For a value $a$ of $A$ with $\Pp(A=a)>0$ and for $z\in G^k$, we write $Q^a_{z}$ and $Q^{\prime a}_{z}$ for the conditional laws of $W(z)$ and of $W'(z)$, respectively, given $A=a$.
The next lemma bounds the second moment of $F-F'$ when, on each event $\{A=a\}$, the elements $W(Z)$ and $W'(Z')$, as well as their conditional laws, differ by the same right translation.

\begin{lemma}\label{lem:transport}
Let $\mathcal V$ be a non-empty set of values of $A$, each of positive probability, let $p=\Pp(A\in\mathcal V)$, and let $c$ be a non-negative function on $\mathcal V$ such that $\sum_{a\in\mathcal V}e^{-c(a)}\leqslant1$.
Suppose that, for each $a\in\mathcal V$, on the event $\{A=a\}$ there is $g\in G$ such that
\begin{equation}\label{eq:transport}
W'(Z')=W(Z)\,g\qquad\text{and}\qquad Q^{\prime a}_{Z'}(xg)=Q^a_{Z}(x)\quad\text{for all }x\in G.
\end{equation}
Then
\[
\E\bigl[(F-F')^2;\,A\in\mathcal V\bigr]\leqslant Cp\Bigl(1+\log^2\frac1p\Bigr)+C\,\E\bigl[c(A)^2;\,A\in\mathcal V\bigr].
\]
\end{lemma}

\begin{proof}
Fix $a\in\mathcal V$ and $z\in G^k$.
Since $W$ is uniformly distributed on $\Words_m$, the law of $W(z)$ is $q_{z,m}$, and hence $q_{z,m}\geqslant\Pp(A=a)\,Q^a_{z}$.
Likewise $q_{z,m}\geqslant\Pp(A=a)\,Q^{\prime a}_{z}$, since $W'$ is also uniformly distributed.
As $Z$ is independent of $(W,W')$, the conditional law of $W(Z)$ given $A=a$ and $Z=z$ is $Q^a_{z}$.
By Lemma~\ref{lem:mixture}(i), we therefore have
\[
\E\biggl[\log^2\frac{q_{Z,m}(W(Z))}{Q^a_{Z}(W(Z))}\biggm|A=a\biggr]\leqslant\log^2\frac1{\Pp(A=a)}+2,
\]
and the same bound holds with $(Z,W,Q^a)$ replaced by $(Z',W',Q^{\prime a})$.
By \eqref{eq:transport}, we have $Q^{\prime a}_{Z'}\bigl(W'(Z')\bigr)=Q^a_{Z}\bigl(W(Z)\bigr)$ on $\{A=a\}$, so the difference of the two logarithms is $F'-F$.
Using the inequality $(x-y)^2\leqslant2x^2+2y^2$, we obtain
\[
\E\bigl[(F-F')^2\mid A=a\bigr]\leqslant4\log^2\frac1{\Pp(A=a)}+8.
\]
Multiplying by $\Pp(A=a)$, summing over $a\in\mathcal V$, and applying Lemma~\ref{lem:entropy-sum}, we obtain the lemma.
\end{proof}

The following lemma produces the right translation required in \eqref{eq:transport}.

\begin{lemma}\label{lem:intervals}
Let $\mathscr I$ be a set of disjoint intervals, let $D\subseteq[k]$, and let $\mathcal P$ be a set of pairs of words in $\Words_m$ with the following properties:
\begin{enumerate}
\item[(i)] the two words of each pair in $\mathcal P$ coincide outside the intervals of $\mathscr I$;
\item[(ii)] for each $I\in\mathscr I$, the restriction $v_I$ is the same for all $(v,v')\in\mathcal P$, and so is $v'_I$;
\item[(iii)] every occurrence of an index of $D$ in a word of a pair in $\mathcal P$ lies in an interval of $\mathscr I$.
\end{enumerate}
Let $z,z'\in G^k$ satisfy $z_i=z'_i$ for all $i\notin D$, and suppose that $v_I(z)=v'_I(z')$ for all $(v,v')\in\mathcal P$ and $I\in\mathscr I$, except possibly for the rightmost interval of $\mathscr I$ when it ends at $m$.
Then there is $g\in G$ such that $v'(z')=v(z)\,g$ for all $(v,v')\in\mathcal P$.
\end{lemma}

\begin{proof}
Let $(v,v')\in\mathcal P$.
The intervals of $\mathscr I$, together with the maximal intervals between them, partition $[m]$, so $v(z)$ and $v'(z')$ are the products, from left to right, of the values of the corresponding restrictions.
On an interval between those of $\mathscr I$, the words $v$ and $v'$ have the same letters and no index of $D$ occurs, so the two values coincide there, as $z_i=z'_i$ for all $i\notin D$.
On an interval $I\in\mathscr I$, the two values are $v_I(z)$ and $v'_I(z')$.
If these coincide for every $I\in\mathscr I$, then $v'(z')=v(z)$, and we take $g=\mathrm{id}$.
Otherwise, the rightmost interval $I$ of $\mathscr I$ ends at $m$, and $v(z)=x\,v_I(z)$ and $v'(z')=x\,v'_I(z')$ for the same $x\in G$.
Hence $v'(z')=v(z)\,g$, where $g=v_I(z)^{-1}v'_I(z')$.
By~(ii), the element $g$ does not depend on the pair.
\end{proof}

We now choose one of the coordinates of $(B,U,W)$, let $(B',U',W')$ be obtained by replacing it with an independent copy, and let
\[
Z=\Psi_W(B,U)\qquad\text{and}\qquad Z'=\Psi_{W'}(B',U').
\]
Note that $W'=W$ unless the chosen coordinate is a letter, in which case $(B',U')=(B,U)$.
By Lemma~\ref{lem:prefix-law}, the random variables $(Z,W,Z',W')$ satisfy the assumptions above, and $F$ and $F'$ are the information content of the walk before and after the replacement.
Now we can bound the effect of changing a single coordinate.

\begin{proposition}\label{prop:local}
Let $k$ and $m$ be integers with $1\leqslant m\leqslant k/4096$, and suppose that one of the following holds:
\begin{enumerate}
\item[(a)] the chosen coordinate is $U_j$, and $E$ is the event that $j\in S(W)$;
\item[(b)] the chosen coordinate is $B_i$, and $E$ is the event that $N_i(W)\geqslant2$;
\item[(c)] the chosen coordinate is $W_j$, and $E$ is the event that $W'_j\neq W_j$.
\end{enumerate}
If $p=\Pp(E)$ is positive, then
\[
\E\bigl[(F-F')^2;\,E\bigr]\leqslant Cp\Bigl(\log^2\frac1p+\log^2(2k)\Bigr).
\]
\end{proposition}

\begin{proof}
We will construct a function $A$ of $(W,W')$ that takes one value outside $E$, such that \eqref{eq:transport} holds for every value $a$ of $A$ taken on $E$ with $\Pp(A=a)>0$.
On $E$, we take $A$ to be a list of data determined by $(W,W')$.
Each entry of the list is a position, an index, a sign, or one of finitely many labels, and takes at most $2km$ values.
We write $|a|$ for the number of entries of a list $a$, and we let $c(a)=|a|\log(4km)$.
As there are at most $(2km)^b$ lists with $b$ entries, the sum of $e^{-c(a)}$ over all non-empty lists $a$ is at most $\sum_{b\geqslant1}(2km)^b(4km)^{-b}=1$.
We will show that
\begin{equation}\label{eq:record-length}
\E\bigl[|A|^2;\,E\bigr]\leqslant Cp.
\end{equation}
Since $4km\leqslant(2k)^2$, by Lemma~\ref{lem:transport}, we then have
\[
\E\bigl[(F-F')^2;\,E\bigr]\leqslant Cp\Bigl(1+\log^2\frac1p\Bigr)+C\log^2(2k)\,\E\bigl[|A|^2;\,E\bigr]\leqslant Cp\Bigl(\log^2\frac1p+\log^2(2k)\Bigr).
\]

Fix a value $a$ of $A$ taken on $E$, with $\Pp(A=a)>0$.
Its \textit{fibre} is the set of pairs $(v,v')$ with $\Pp\bigl((W,W')=(v,v')\bigr)>0$ and $A(v,v')=a$.
In each case, we specify a set $\mathscr I$ of disjoint intervals and a set $D\subseteq[k]$, both determined by $a$.
We then check the hypotheses of Lemma~\ref{lem:intervals} on the event $\{A=a\}$, with the fibre as $\mathcal P$ and with $(z,z')=(Z,Z')$.
Lemma~\ref{lem:intervals} then provides $g\in G$ such that $v'(Z')=v(Z)\,g$ for every pair $(v,v')$ in the fibre, and in particular $W'(Z')=W(Z)\,g$.
Since $(W,W')$ takes values in the fibre given $A=a$, it follows that, for $(z,z')=(Z,Z')$, the conditional law of $W'(z')$ given $A=a$ is the image of the conditional law of $W(z)$ given $A=a$ under the right translation $x\mapsto xg$.
That is, $Q^{\prime a}_{Z'}(xg)=Q^a_{Z}(x)$ for all $x\in G$, which is \eqref{eq:transport}.

In each case, $A$ records all occurrences of the indices whose generators change, because the identities between products required in Lemma~\ref{lem:intervals} must hold for every pair in the fibre, and not only for the realised pair.

In the three cases below, $w$ denotes a word with $S(w)=\{j_1<\dots<j_s\}$, and we let $j_0=0$.

\textit{Case~(a).}
Here $W'=W$ and $B'=B$.
Let $w$ be a word with $j\in S(w)$, say $j=j_\ell$.
By \eqref{eq:prefix-coordinates}, replacing $U_j$ changes only the generators at the positions $j_\ell$ and, when $\ell<s$, $j_{\ell+1}$.

Suppose first that $\ell<s$.
Let $A$ record the singleton position $j_{\ell+1}$ following $j$, the restriction $w_{(j-1,\,j_{\ell+1}]}$, and the occurrences of the indices $i_j$ and $i_{j_{\ell+1}}$, each of which occurs once.
We take $\mathscr I=\{(j-1,\,j_{\ell+1}]\}$ and $D=\{i_j,i_{j_{\ell+1}}\}$.
Every pair in the fibre has the form $(v,v)$, where $v$ has the recorded restriction.
By the recorded occurrences, the two indices of $D$ occur in $v$ only at $j$ and $j_{\ell+1}$, and $Z$ and $Z'$ agree outside $D$.
By \eqref{eq:prefix}, at both $Z$ and $Z'$, the products of the first $j-1$ letters and of the first $j_{\ell+1}$ letters of $W$ are
\[
U_{j_{\ell-1}}\,w_{(j_{\ell-1},\,j-1]}(B)\qquad\text{and}\qquad U_{j_{\ell+1}},
\]
respectively.
Neither expression involves $U_j$, so the products over $(j-1,\,j_{\ell+1}]$ coincide.

Suppose next that $\ell=s$.
Let $A$ record the fact that $j$ is the last singleton position, the restriction $w_{(j-1,\,m]}$, and the occurrence of $i_j$.
We take $\mathscr I=\{(j-1,m]\}$ and $D=\{i_j\}$.
As before, the index $i_j$ occurs only at $j$ in every word of the fibre, and $Z$ and $Z'$ agree outside $D$.
Since the interval ends at $m$, no identity between the products is required.

In both subcases, the interval in $\mathscr I$ is contained in the block of $S(w)\setminus\{j\}$ containing $j$, and $A$ records at most five entries besides the letters of $w$ in this interval.
Conditional on $W_j$ and on $E$, the letters of $W$ at the other positions are independent, with indices uniformly distributed on $[k]\setminus\{i_j\}$.
By Lemma~\ref{lem:blocks} with $T=\{j\}$, it follows that $\E\bigl[|A|^2\bigm|W_j,\,E\bigr]\leqslant C$, which implies \eqref{eq:record-length}.

\textit{Case~(b).}
Here $W'=W$ and $U'=U$.
Let $w$ be a word with $N_i(w)\geqslant2$.
We call a block of $S(w)$ \textit{marked} if it contains an occurrence of $i$.
The index $i$ is not a singleton, so its generator is $B_i$.
By \eqref{eq:prefix-coordinates}, the coordinate $B_i$ enters the definition of the generator at a singleton position $j_\ell$ precisely when $i$ occurs in $w_{(j_{\ell-1},\,j_\ell-1]}$, that is, when the block $(j_{\ell-1},j_\ell]$ is marked.
Therefore $Z$ and $Z'$ agree outside the set $D$ consisting of $i$ and of the indices at the singleton positions that end marked blocks.

Let $A$ record the marked blocks, the restrictions of $w$ to them, and the occurrences of $i$ and of the indices at these singleton positions.
We take $\mathscr I$ to be the set of marked blocks.
Every pair in the fibre has the form $(v,v)$, where $v$ has the recorded restrictions.
By the recorded occurrences, every occurrence of an index of $D$ in $v$ lies in a marked block.
For a marked block $(j_{\ell-1},j_\ell]$, it follows from \eqref{eq:prefix} that
\[
w_{(j_{\ell-1},j_\ell]}(Z)=U_{j_{\ell-1}}^{-1}U_{j_\ell}=w_{(j_{\ell-1},j_\ell]}(Z'),
\]
since $U$ is unchanged.
The products may differ only on the final block $(j_s,m]$, if it is marked, and this block ends at $m$, as Lemma~\ref{lem:intervals} allows.
Let $T=\{r\in[m]:i_r=i\}$, so that $|T|=N_i(w)$.
Since no position of $T$ is a singleton position, the marked blocks are the blocks of $S(w)\setminus T$ that contain a position of $T$, and besides the letters of $w$ in these blocks, $A$ records at most $6|T|$ entries.
Conditional on the occurrences of $i$ in $W$, the letters of $W$ at the remaining positions are independent, with indices uniformly distributed on $[k]\setminus\{i\}$.
By Lemma~\ref{lem:blocks}, it follows that $\E\bigl[|A|^2;\,E\bigr]\leqslant C\,\E\bigl[N_i(W)^4;\,N_i(W)\geqslant2\bigr]$.
Since $N_i(W)$ is a binomial random variable with parameters $m$ and $1/k$, and $m/k\leqslant1/4096$, the right-hand side is at most $C\,\Pp\bigl(N_i(W)\geqslant2\bigr)=Cp$, which is \eqref{eq:record-length}.

\textit{Case~(c).}
Here $(B',U')=(B,U)$, and $W'$ differs from $W$ at most in its $j$th letter.
Let $(w,w')$ be a pair of words that differ at most at position $j$, and write $w_j=(i_j,\varepsilon_j)$ and $w'_j=(i'_j,\varepsilon'_j)$.
Only the multiplicities of $i_j$ and $i'_j$ differ between $w$ and $w'$.
Hence a position other than $j$ can be a singleton position of $w'$ but not of $w$ only if it is the other occurrence of $i_j$ and $N_{i_j}(w)=2$.
Similarly, it can be a singleton position of $w$ but not of $w'$ only if it is the occurrence of $i'_j$ in $w$, $N_{i'_j}(w)=1$, and $i'_j\neq i_j$.
In particular, at most three positions belong to exactly one of $S(w)$ and $S(w')$.

Let $R$ be the set of positions other than $j$ that are singleton positions of both $w$ and $w'$.
We call a block of $R$ \textit{marked} if it contains $j$ or a position belonging to exactly one of $S(w)$ and $S(w')$, so there are at most four marked blocks.
Let $D$ consist of those of the indices $i_j$ and $i'_j$ that are singletons of $w$ or of $w'$, together with the indices at the elements of $R$ that end marked blocks.
We claim that $Z$ and $Z'$ agree outside $D$.
Let $i\notin D$.
If $i$ is a singleton of neither word, then both generators at $i$ equal $B_i$.
Otherwise, since $i\notin D$, we have $i\notin\{i_j,i'_j\}$, so $i$ has the same multiplicity in both words and is a singleton of both.
Its position lies in $R$, and the block ending there is not marked.
This block does not contain $j$, and contains no singleton position of $w$ or of $w'$ other than its right endpoint.
Its left endpoint is therefore the singleton position preceding the position of $i$ in both words, or $0$, and $w$ and $w'$ coincide on the block.
By \eqref{eq:prefix-coordinates}, the two generators at $i$ are given by the same expression in $(B,U)$, which proves the claim.

Every occurrence of an index of $D$ in $w$ or $w'$ lies in a marked block.
Indeed, an index at an element of $R$ occurs only there, in both words.
If $i_j$ is a singleton of $w$, then its only occurrence in $w$ is at $j$, and it occurs in $w'$ at most at $j$.
If $i_j$ is a singleton of $w'$ but not of $w$, then $N_{i_j}(w)=2$.
Its other occurrence is then a singleton position of $w'$ but not of $w$, and hence lies in a marked block.
The same argument applies to $i'_j$, with the roles of $w$ and $w'$ interchanged.
Now let $(r,r']$ be a marked block with $r'\in R$.
Then $r$ is $0$ or lies in $R$, and by \eqref{eq:prefix} we have
\[
w_{(r,r']}(Z)=U_r^{-1}U_{r'}=w'_{(r,r']}(Z').
\]
Hence the products coincide on every marked block except possibly the final block, which ends at $m$.

Let $A$ record $j$, the letters $w_j$ and $w'_j$, the marked blocks, and the restrictions of $w$ and $w'$ to them.
Let it also record which right endpoints of the marked blocks lie in $R$, and the occurrences in both words of $i_j$, of $i'_j$, and of the indices at these right endpoints.
These data determine $D$, and we take $\mathscr I$ to be the set of marked blocks.
Every pair in the fibre consists of two words that differ at most at position $j$, which lies in a marked block, and has the recorded restrictions.
By the recorded occurrences, all occurrences of indices of $D$ in these words lie in marked blocks.
Thus the hypotheses of Lemma~\ref{lem:intervals} hold.

It remains to prove \eqref{eq:record-length} in this case.
Let $T=\{r\in[m]:i_r\in\{i_j,i'_j\}\}$.
We claim that $R=S(w)\setminus T$.
Indeed, if $r\notin T$, then $i_r$ has the same multiplicity in $w$ and $w'$, so $r\in R$ if and only if $r\in S(w)$.
If $r\in T\setminus\{j\}$, then $i_r$ occurs at both $r$ and $j$ in $w$ or in $w'$, so $r\notin R$, and $j\notin R$ by definition.
The positions belonging to exactly one of $S(w)$ and $S(w')$ lie in $T$, so every marked block is a block of $S(w)\setminus T$ that contains a position of $T$.
Besides the letters of $w$ and $w'$ in the marked blocks, $A$ records at most $C|T|$ entries.
Conditional on $W_j$, $W'_j$, the set $T$, and the letters of $W$ at the positions in $T$, the letters of $W$ at the remaining positions are independent, with indices uniformly distributed on $[k]\setminus\{i_j,i'_j\}$.
Hence, by Lemma~\ref{lem:blocks}, we have $\E\bigl[|A|^2;\,E\bigr]\leqslant C\,\E\bigl[|T|^4\bigr]$.
Given $W_j$ and $W'_j$, the number $|T|-1$ is a binomial random variable with parameters $m-1$ and $|\{i_j,i'_j\}|/k\leqslant2/k$, so $\E[|T|^4]\leqslant C$.
Since $p=1-1/|\alphabet|\geqslant1/2$, this proves \eqref{eq:record-length}, which completes the proof.
\end{proof}

\section{Variance of the information content}\label{sec:variance}

In this section, we prove Theorem~\ref{thm:information}.
Our main tool is the Efron--Stein inequality \cite{ES81,Ste86}; see also \cite[Chapter~3]{BLM13}.
Let $F$ be a square-integrable function of finitely many independent random variables.
For each of these variables, let $F'$ denote the value of $F$ when that variable is replaced with an independent copy.
Then
\begin{equation}\label{eq:efron-stein}
\Var F\leqslant\frac12\sum\E\bigl[(F-F')^2\bigr],
\end{equation}
where the sum is over the variables.
Eberhard and Varj\'u \cite[Section~2]{EV21} used this inequality to bound the variance of the information content for the Chung--Diaconis--Graham process \cite{CDG87}, and Huang and Rao \cite[Section~3.4.2]{HR25} did the same for an auxiliary process on $\Z^k$.

\begin{proof}[Proof of Theorem~\ref{thm:information}]
Let $B$, $U$, and $W$ be as in Section~\ref{sec:coordinates}, and let $F=-\log q_{Z,m}\bigl(W(Z)\bigr)$, where $Z=\Psi_W(B,U)$.
By Lemma~\ref{lem:prefix-law}, the pair $(Z,W)$ has the joint law in Theorem~\ref{thm:information}, so $F$ has the law of $\Info_m$.
We write $F^{B_i}$, $F^{U_j}$, and $F^{W_j}$ for the values of $F$ after the indicated coordinate is replaced with an independent copy.
As in Section~\ref{sec:change}, we write $(B',U',W')$ for the coordinates after the replacement and $Z'=\Psi_{W'}(B',U')$.
We bound the terms of \eqref{eq:efron-stein} using Proposition~\ref{prop:local} and the inequality $p\log^2(1/p)\leqslant1$ for $p\in(0,1]$.

\textit{Replacing $U_j$.}
If $j\notin S(W)$, then $U_j$ is not used, and $F^{U_j}=F$.
By case~(a) of Proposition~\ref{prop:local}, the contribution to $\E\bigl[(F-F^{U_j})^2\bigr]$ of the event that $j\in S(W)$ is at most $C\log^2(2k)$.
Summing over $j$, we obtain
\begin{equation}\label{eq:U-energy}
\sum_{j=1}^m\E\bigl[(F-F^{U_j})^2\bigr]\leqslant Cm\log^2(2k).
\end{equation}

\textit{Replacing $W_j$.}
If $W'_j=W_j$, then $F^{W_j}=F$.
By case~(c) of Proposition~\ref{prop:local}, the contribution of the event that $W'_j\neq W_j$ is at most $C\log^2(2k)$, and hence
\begin{equation}\label{eq:W-energy}
\sum_{j=1}^m\E\bigl[(F-F^{W_j})^2\bigr]\leqslant Cm\log^2(2k).
\end{equation}

\textit{Replacing $B_i$.}
If $i$ is a singleton of $W$, then $B_i$ is not used, and $F^{B_i}=F$.
We treat separately the indices that do not occur in $W$ and those that occur at least twice.
For an index $i$ that does not occur in $W$, replacing $B_i$ does not move the walk, but it changes the law $q_{Z,m}$, since other words may use the generator $Z_i$.
Since there are at least $k-m$ such indices, we need a bound of order $m/k$ for each of them.

Fix $i\in[k]$, and let $\alpha=(1-1/k)^m$ be the probability that $i$ does not occur in $W$.
Note that $\alpha\geqslant1-m/k\geqslant1/2$.
For $z\in G^k$, let $\nu_{z}$ and $\nu'_{z}$ be the conditional laws of $W(z)$ given $N_i(W)=0$ and given $N_i(W)\geqslant1$, respectively.
Then $q_{z,m}=\alpha\nu_{z}+(1-\alpha)\nu'_{z}$, and $\nu_{z}$ does not depend on $z_i$.
On the event $\{N_i(W)=0\}$, replacing $B_i$ changes only the $i$th coordinate of the tuple, so $W(Z')=W(Z)$ and $\nu_{Z'}=\nu_{Z}$ there, and
\[
F-F^{B_i}=\log\frac{q_{Z',m}(W(Z'))}{\nu_{Z'}(W(Z'))}-\log\frac{q_{Z,m}(W(Z))}{\nu_{Z}(W(Z))}.
\]
By Lemma~\ref{lem:prefix-law}, conditional on $N_i(W)=0$ and $Z=z$, the element $W(Z)$ has law $\nu_{z}$, and the same holds with $Z'$ in place of $Z$.
By Lemma~\ref{lem:mixture}(ii), the conditional second moment of each of the two logarithms is therefore at most $3(1-\alpha)\leqslant3m/k$, and we obtain
\begin{equation}\label{eq:B-absent}
\E\bigl[(F-F^{B_i})^2;\,N_i(W)=0\bigr]\leqslant\frac{12m}k.
\end{equation}
Summing \eqref{eq:B-absent} over $i$, we find that the indices that do not occur in $W$ contribute at most $12m$.

Now consider the indices occurring at least twice, and let $p_i=\Pp\bigl(N_i(W)\geqslant2\bigr)$.
If $p_i>0$, then, by case~(b) of Proposition~\ref{prop:local}, the contribution of the event that $N_i(W)\geqslant2$ to $\E\bigl[(F-F^{B_i})^2\bigr]$ is at most $Cp_i\bigl(\log^2(1/p_i)+\log^2(2k)\bigr)$.
We have $\sum_ip_i\leqslant\sum_i\E[N_i(W)]/2=m/2$ and $p_i\leqslant\binom m2k^{-2}\leqslant m/k\leqslant1/4096$.
Since $x\mapsto x\log^2(1/x)$ is increasing on $(0,e^{-2}]$ and there are $k$ indices, it follows that
\[
\sum_{i:p_i>0}p_i\Bigl(\log^2\frac1{p_i}+\log^2(2k)\Bigr)\leqslant m\log^2\frac km+m\log^2(2k)\leqslant2m\log^2(2k).
\]
Combining these bounds, we obtain
\begin{equation}\label{eq:B-energy}
\sum_{i=1}^k\E\bigl[(F-F^{B_i})^2\bigr]\leqslant Cm\log^2(2k).
\end{equation}

Substituting \eqref{eq:U-energy}, \eqref{eq:W-energy}, and \eqref{eq:B-energy} into \eqref{eq:efron-stein}, we conclude that $\Var(\Info_m)=\Var F\leqslant Cm\log^2(2k)$, as required.
\end{proof}

\section{Proof of the main theorem}\label{sec:cutoff}

In this section, we prove Theorems~\ref{thm:main} and~\ref{thm:window}.
First, we record an $L_2$ estimate and a spectral estimate, and we bound the times $t_\omega$.
When $k/\log n\to\infty$, we also show that the entropic time lies between two classical bounds, which already yield cutoff when $\log k>(\log\log n)^2$.
Finally, we prove a continuous-time counterpart of Theorem~\ref{thm:information}, and we combine it with the spectral estimate to prove Theorem~\ref{thm:window}, from which we deduce Theorem~\ref{thm:main}.

For a function $f$ on $G$, we write $\pi(f)=\sum_xf(x)\pi(x)$ and $\|f\|_{2,\pi}=\pi(f^2)^{1/2}$.
The \textit{$L_2$ distance} of a probability measure $\mu$ on $G$ from $\pi$ is $\|\mu-\pi\|_2=\|n\mu-1\|_{2,\pi}$, so that $\|\mu-\pi\|_2^2=n\sum_x\mu(x)^2-1$.
It is related to the total variation distance and to the entropy by
\begin{equation}\label{eq:tv-entropy-chi}
2\|\mu-\pi\|_{\TV}\leqslant\|\mu-\pi\|_2\qquad\text{and}\qquad\log n-\Ent(\mu)\leqslant\|\mu-\pi\|_2^2.
\end{equation}
The first inequality follows from the Cauchy--Schwarz inequality.
The second follows by applying the inequality $\log u\leqslant u-1$ to each term of $\log n-\Ent(\mu)=\sum_x\mu(x)\log\bigl(n\mu(x)\bigr)$.
The following estimate is essentially due to Roichman \cite[Theorem~2]{R96}; see also \cite[Section~7.2]{HOT26}.

\begin{lemma}\label{lem:collision}
Let $G$ be a finite group of order $n$, and let $k$ and $m$ be integers with $1\leqslant m\leqslant k$.
Then
\[
\E\,\|q_{Z,m}-\pi\|_2^2\leqslant(n-1)\binom km\Bigl(\frac mk\Bigr)^{2m}\leqslant n\Bigl(\frac{em}k\Bigr)^m.
\]
\end{lemma}

\begin{proof}
Let $W$ and $W'$ be chosen uniformly at random from $\Words_m$, with $Z$, $W$, and $W'$ independent.
Conditional on $Z$, the values $W(Z)$ and $W'(Z)$ are independent with law $q_{Z,m}$.
Hence
\[
\E\,\|q_{Z,m}-\pi\|_2^2=n\,\Pp\bigl(W(Z)=W'(Z)\bigr)-1.
\]
Suppose that some index $i$ occurs exactly once among the $2m$ letters of $W$ and $W'$.
We condition on $W$, $W'$, and the coordinates of $Z$ other than $Z_i$.
Since $Z_i$ occurs exactly once in the equation $W(Z)=W'(Z)$, the equation holds for exactly one value of $Z_i$, so its conditional probability is $1/n$.
Otherwise, every index occurring among the $2m$ letters occurs at least twice.
Then all the letters have indices in some $m$-subset of $[k]$.
By a union bound over these subsets, this happens with probability at most $\binom km(m/k)^{2m}$.
Bounding the conditional probability of the equation by one in this case, we obtain
\[
\Pp\bigl(W(Z)=W'(Z)\bigr)\leqslant\frac1n+\Bigl(1-\frac1n\Bigr)\binom km\Bigl(\frac mk\Bigr)^{2m},
\]
and the first inequality follows.
The second follows from $\binom km\leqslant(ek/m)^m$.
\end{proof}

For $z\in G^k$, let
\[
\lambda_*(z)=\max\bigl\{\|P_{z}f\|_{2,\pi}:\pi(f)=0\text{ and }\|f\|_{2,\pi}=1\bigr\}
\]
be the norm of $P_{z}$ on the orthogonal complement of the constant functions.
For the undirected walk, $P_{z}$ is symmetric, so $\lambda_*(z)$ is the largest absolute value of an eigenvalue of $P_{z}$ on this space, and $1-\lambda_*(z)$ is the absolute spectral gap of $P_{z}$; see \cite[Section~12.2]{LP17}.
In particular, $\lambda_*(z)=1$ if $\Cay(G,z)$ is disconnected.
We write $P_{z}^{\top}$ for the transpose of $P_{z}$.
Since $\pi$ is uniform, $P_{z}^{\top}$ is the adjoint of $P_{z}$ in $L_2(\pi)$, and since both matrices fix the constant functions, they preserve the orthogonal complement of the constant functions.
Hence $\|P_{z}^{\top}f\|_{2,\pi}\leqslant\lambda_*(z)\|f\|_{2,\pi}$ whenever $\pi(f)=0$.
The following form of the theorem of Alon and Roichman \cite[Theorem~1]{AR94} is proved in the same way as Lemma~\ref{lem:collision}.

\begin{lemma}\label{lem:spectrum}
Let $G$ be a finite group of order $n\geqslant2$, and let $k\geqslant2^{14}\log n$.
Then
\[
\Pp\Bigl(\lambda_*(Z)^2>\frac{e^3\lceil\log n\rceil}k\Bigr)\leqslant\frac1n.
\]
\end{lemma}

\begin{proof}
Let $z\in G^k$, and let $m=\lceil\log n\rceil$, so that $m\leqslant k$.
The matrix $(P_{z}^{\top}P_{z})^m$ is symmetric and positive semidefinite, it fixes the constant functions, and its norm on their orthogonal complement is $\lambda_*(z)^{2m}$.
As its eigenvalues are non-negative, it follows that
\[
\lambda_*(z)^{2m}\leqslant\tr\bigl((P_{z}^{\top}P_{z})^m\bigr)-1.
\]
Let $W$ and $W'$ be chosen uniformly at random from $\Words_m$, with $Z$, $W$, and $W'$ independent, and regard each of their letters as a word of length $1$.
Since $P_{z}^{\top}(x,y)=P_{z}(y,x)$ and $P_{z}(gx,gy)=P_{z}(x,y)$ for all $g,x,y\in G$, we have
\[
\tr\bigl((P_{z}^{\top}P_{z})^m\bigr)=n\,\Pp\bigl(W_1(z)^{-1}W'_1(z)\cdots W_m(z)^{-1}W'_m(z)=\mathrm{id}\bigr).
\]
Suppose that some index $i$ occurs exactly once among the $2m$ letters of $W$ and $W'$.
We condition on $W$, $W'$, and the coordinates of $Z$ other than $Z_i$.
Since $Z_i$ occurs exactly once in the equation
\[
W_1(Z)^{-1}W'_1(Z)\cdots W_m(Z)^{-1}W'_m(Z)=\mathrm{id},
\]
this equation holds for exactly one value of $Z_i$, so its conditional probability is $1/n$.
Otherwise, every index occurring among the $2m$ letters occurs at least twice, so all the letters have indices in some $m$-subset of $[k]$, which happens with probability at most $\binom km(m/k)^{2m}\leqslant(em/k)^m$.
Hence $\E\,\tr\bigl((P_{Z}^{\top}P_{Z})^m\bigr)-1\leqslant n(em/k)^m$, and, by Markov's inequality,
\[
\Pp\Bigl(\lambda_*(Z)^2>\frac{e^3m}k\Bigr)\leqslant\Bigl(\frac k{e^3m}\Bigr)^mn\Bigl(\frac{em}k\Bigr)^m=ne^{-2m}\leqslant\frac1n,
\]
since $m\geqslant\log n$.
\end{proof}

Next, we record two elementary properties of the entropy of the walk.

\begin{lemma}\label{lem:entropy-increments}
Let $z\in G^k$, and let $m\geqslant0$ be an integer.
Then
\[
0\leqslant\Ent(q_{z,m+1})-\Ent(q_{z,m})\leqslant\log(2k).
\]
Consequently, the sequence $(h_m)_{m\geqslant0}$ is non-decreasing, and $0\leqslant h_m\leqslant\min\{m\log(2k),\log n\}$ for every $m$.
\end{lemma}

\begin{proof}
The measure $q_{z,m+1}$ is the average of the right translates of $q_{z,m}$ by the elements $z_i^{\varepsilon}$ with $(i,\varepsilon)\in\alphabet$, each of which has entropy $\Ent(q_{z,m})$.
The lower bound therefore follows from the concavity of entropy.
If $X$ has law $q_{z,m}$ and $Y$ is a letter chosen uniformly at random from $\alphabet$, independently of $X$, then $X\,Y(z)$ has law $q_{z,m+1}$ and is a function of $(X,Y)$.
Hence $\Ent(q_{z,m+1})\leqslant\Ent(q_{z,m})+\log|\alphabet|\leqslant\Ent(q_{z,m})+\log(2k)$.
The remaining assertions follow, since $\Ent(q_{z,0})=0$ and every probability measure on $G$ has entropy at most $\log n$.
\end{proof}

The next lemma bounds the times $t_\omega$ defined in Section~\ref{sec:intro}.

\begin{lemma}\label{lem:omega-bounds}
Let $G$ be a finite group of order $n\geqslant2$, let $k\geqslant2^{14}\log n$, and let $\omega\geqslant1$.
Then
\[
\frac{\log n-2\omega}{\log(2k)}\leqslant t_\omega\leqslant\lceil\log n\rceil.
\]
\end{lemma}

\begin{proof}
The lower bound follows from the inequality $h_m\leqslant m\log(2k)$ of Lemma~\ref{lem:entropy-increments}.
For the upper bound, let $m=\lceil\log n\rceil$.
Since $\log n\geqslant\log2>1/2$, we have $m\leqslant2\log n\leqslant k/8192$, and hence, by \eqref{eq:tv-entropy-chi} and Lemma~\ref{lem:collision},
\[
\log n-h_m\leqslant\E\,\|q_{Z,m}-\pi\|_2^2\leqslant n\Bigl(\frac{em}k\Bigr)^m\leqslant n\Bigl(\frac e{8192}\Bigr)^{\log n}\leqslant n^{-7}.
\]
Thus $h_m\geqslant\log n-2\omega$, and $t_\omega\leqslant m$, as required.
\end{proof}

For the remainder of this section, $G$ ranges over a sequence of finite groups of order $n\to\infty$, the integers $k=k(G)$ satisfy \eqref{eq:regime}, and $Z=(Z_1,\dots,Z_k)$, where $Z_1,\dots,Z_k$ are chosen independently and uniformly at random from $G$.
Recall the two classical bounds on the mixing time when $k/\log n\to\infty$: the lower bound $\log n/\log k$ comes from the diameter, since before this time the walk can reach only a vanishing proportion of the vertices.
The upper bound $\log n/\log(k/\log n)$ is due to Dou and Hildebrand \cite[Theorem~1]{DH96} and Roichman \cite[Theorems~1 and~2]{R96}; see \cite[Section~1.3.1]{HOT26}.
In this case, we have $\log(k/\log n)\to\infty$ and $\log k=o(\log n)$, so both bounds tend to infinity, and the upper bound is at least the lower bound.
For completeness, we derive both bounds in our setting, and the same computation shows that the expected entropy is close to $\log n$ by time $(1+o(1))\log n/\log(k/\log n)$.

\begin{proposition}\label{prop:time-bounds}
Suppose, in addition, that $k/\log n\to\infty$, and let $\delta\in(0,1)$ be fixed.
Then
\[
\inf_{z\in G^k}d_{z}\Bigl((1-\delta)\frac{\log n}{\log k}\Bigr)\to1\qquad\text{and}\qquad d_{Z}\Bigl((1+\delta)\frac{\log n}{\log(k/\log n)}\Bigr)\inP0.
\]
Moreover, for every fixed $\eta>0$,
\begin{equation}\label{eq:entropy-near-uniform}
\log n-h_m\to0\qquad\text{uniformly in }m\geqslant(1+\eta)\frac{\log n}{\log(k/\log n)}.
\end{equation}
\end{proposition}

\begin{proof}
Let $t=(1-\delta)\log n/\log k$.
The number of steps of the walk by time $t$ is $\lfloor t\rfloor$ in discrete time, and a Poisson random variable with mean $t$ in continuous time.
Since $t\to\infty$, by Chebyshev's inequality, in both cases this number is at most $(1-\delta/2)\log n/\log k$ with probability tending to one.
At most $2(2k)^{(1-\delta/2)\log n/\log k}$ vertices can be reached from $\mathrm{id}$ in at most this many steps, and
\[
\log\Bigl(2(2k)^{(1-\delta/2)\log n/\log k}\Bigr)=\log2+\Bigl(1-\frac\delta2\Bigr)\Bigl(1+\frac{\log2}{\log k}\Bigr)\log n\leqslant\Bigl(1-\frac\delta3\Bigr)\log n
\]
for large $n$, since $\log k\to\infty$.
Hence, uniformly in $z$, the walk at time $t$ lies with probability tending to one in a set of at most $n^{1-\delta/3}$ vertices, whose uniform measure tends to zero.
This proves the first part.

For the second part and \eqref{eq:entropy-near-uniform}, fix $\eta>0$, and let $m$ be the smallest integer such that $m\log(k/\log n)\geqslant(1+\eta)\log n$.
Then $m=o(k)$, and
\[
\log\frac k{em}=\log\frac k{\log n}+\log\log\frac k{\log n}-\log(1+\eta)-1+o(1)\geqslant\log\frac k{\log n}
\]
for large $n$, since $\log(k/\log n)\to\infty$.
Hence $m\log(k/(em))\geqslant(1+\eta)\log n$, and by Lemma~\ref{lem:collision} we obtain
\[
\E\,\|q_{Z,m}-\pi\|_2^2\leqslant n\Bigl(\frac{em}k\Bigr)^m\leqslant n^{-\eta}.
\]
By \eqref{eq:tv-entropy-chi}, it follows that $\E\,\|q_{Z,m}-\pi\|_{\TV}\to0$ and $\log n-h_m\to0$.
Since $(h_m)$ is non-decreasing by Lemma~\ref{lem:entropy-increments}, the convergence $\log n-h_m\to0$ holds uniformly for all larger $m$, which is \eqref{eq:entropy-near-uniform}.
Finally, for every $z$, the distance $\|q_{z,j}-\pi\|_{\TV}$ is non-increasing in $j$.
Let $t=(1+\delta)\log n/\log(k/\log n)$ and $0<\eta<\delta$.
In discrete time, we have $\lfloor t\rfloor\geqslant m$ for large $n$, and hence $\E\,d_{Z}(t)\to0$.
In continuous time, by \eqref{eq:poisson-law}, $d_{z}(t)$ is at most $\|q_{z,m}-\pi\|_{\TV}$ plus the probability that a Poisson random variable with mean $t$ is less than $m$, and this probability tends to zero by Chebyshev's inequality, so again $\E\,d_{Z}(t)\to0$.
The second part now follows from Markov's inequality.
\end{proof}

As a consequence, the entropic time lies between these two bounds.

\begin{lemma}\label{lem:centre-bounds}
Suppose, in addition, that $k/\log n\to\infty$, and let $\eta>0$ be fixed.
Then, for all sufficiently large $n$,
\[
\frac{\log n-2(\log n)^{2/3}}{\log(2k)}\leqslant t_*\leqslant(1+\eta)\frac{\log n}{\log(k/\log n)}+1.
\]
\end{lemma}

\begin{proof}
The lower bound holds by Lemma~\ref{lem:omega-bounds}.
By \eqref{eq:entropy-near-uniform}, for large $n$, every integer $m$ in the range of \eqref{eq:entropy-near-uniform} satisfies $h_m\geqslant\log n-2(\log n)^{2/3}$, so $t_*$ satisfies the upper bound.
\end{proof}

When $\log k>(\log\log n)^2$, the two bounds are asymptotically equivalent, and Proposition~\ref{prop:time-bounds} already yields cutoff at time $t_*$.

\begin{lemma}\label{lem:large-range}
Suppose, in addition, that $\log k>(\log\log n)^2$.
Then $t_*\sim\log n/\log k\sim\log n/\log(k/\log n)$, and the walk exhibits cutoff with high probability at time $t_*$.
\end{lemma}

\begin{proof}
Since $\log\log n/\log k<1/\log\log n\to0$, we have $k/\log n\to\infty$ and $\log(k/\log n)/\log k=1-\log\log n/\log k\to1$, so the two bounds are asymptotically equivalent.
By Lemma~\ref{lem:centre-bounds}, for every fixed $\eta>0$,
\[
(1-o(1))\frac{\log n}{\log k}\leqslant t_*\leqslant(1+\eta+o(1))\frac{\log n}{\log(k/\log n)},
\]
and letting $\eta\to0$ we obtain the first part.
Now fix $\delta\in(0,1)$.
For large $n$, we have $(1-\delta)t_*\leqslant(1-\delta/2)\log n/\log k$ and $(1+\delta)t_*\geqslant(1+\delta/2)\log n/\log(k/\log n)$.
Since $d_{z}$ is non-increasing, Proposition~\ref{prop:time-bounds}, applied with $\delta/2$ in place of $\delta$, shows that $d_{Z}\bigl((1-\delta)t_*\bigr)\inP1$ and $d_{Z}\bigl((1+\delta)t_*\bigr)\inP0$.
\end{proof}

For $t\geqslant0$, let $X_t$ be a random element of $G$ whose conditional law given $Z$ is $\mu_{Z,t}$.
The next lemma is a continuous-time counterpart of Theorem~\ref{thm:information}.

\begin{lemma}\label{lem:continuous}
Let $t$ be a real number with $1\leqslant t\leqslant k/8192$.
Then
\[
\E\bigl[\bigl(-\log\mu_{Z,t}(X_t)-h_{\lfloor t\rfloor}\bigr)^2\bigr]\leqslant Ct\log^2(2k).
\]
\end{lemma}

\begin{proof}
For the discrete-time walk, $-\log\mu_{Z,t}(X_t)$ has the law of $\Info_{\lfloor t\rfloor}$, whose mean is $h_{\lfloor t\rfloor}$, and the lemma follows from Theorem~\ref{thm:information}, since $1\leqslant\lfloor t\rfloor\leqslant k/4096$.

For the continuous-time walk, let $N$ be a Poisson random variable with mean $t$, independent of $Z$, and, conditional on $N=m$ and $Z$, let $X$ have law $q_{Z,m}$.
By \eqref{eq:poisson-law}, the conditional law of $X$ given $Z$ is $p_{Z,t}$, so $(Z,X)$ has the law of $(Z,X_t)$.
We write
\[
-\log p_{Z,t}(X)-h_{\lfloor t\rfloor}=\log\frac{q_{Z,N}(X)}{p_{Z,t}(X)}+\bigl(-\log q_{Z,N}(X)-h_N\bigr)+\bigl(h_N-h_{\lfloor t\rfloor}\bigr),
\]
and we bound the second moments of the three terms on the right.

For every $z$ and every integer $m\geqslant0$, we have $p_{z,t}\geqslant\Pp(N=m)\,q_{z,m}$.
By Lemma~\ref{lem:mixture}(i), conditional on $N=m$ and $Z=z$, the second moment of the first term is therefore at most $\log^2(1/\Pp(N=m))+2$.
By Stirling's formula, $\log m!=m\log m-m+O(\log(m+2))$, and together with the inequality $\log u\leqslant u-1$, we obtain
\[
-\log\Pp(N=m)=t-m+m\log\frac mt+O\bigl(\log(m+2)\bigr)\leqslant\frac{(m-t)^2}{t}+O\bigl(\log(m+2)\bigr).
\]
Since $\E[(N-t)^4]=3t^2+t\leqslant4t^2$ and $\log^2(N+2)\leqslant N+2$, the second moment of the first term is at most $Ct$.

Conditional on $N=m$, the random variable $-\log q_{Z,m}(X)$ has the law of $\Info_m$, whose mean is $h_m$.
If $1\leqslant m\leqslant2t$, then $m\leqslant k/4096$, and, by Theorem~\ref{thm:information}, the variance of $\Info_m$ is at most $Cm\log^2(2k)$, while $\Info_0=0$.
For larger $m$, we use the bound $0\leqslant\Info_m\leqslant m\log(2k)$, which holds since $q_{z,m}(w(z))\geqslant(2k)^{-m}$ for every $z\in G^k$ and every $w\in\Words_m$.
Furthermore, $\Pp(N>2t)\leqslant\E[(N-t)^4]/t^4\leqslant4/t^2$, and hence $\E[N^2;\,N>2t]\leqslant(\E N^4)^{1/2}\,\Pp(N>2t)^{1/2}\leqslant Ct$.
It follows that the second moment of the second term is at most $Ct\log^2(2k)$.

Finally, by Lemma~\ref{lem:entropy-increments}, we have $|h_N-h_{\lfloor t\rfloor}|\leqslant|N-\lfloor t\rfloor|\log(2k)$, and $\E[(N-\lfloor t\rfloor)^2]\leqslant t+1\leqslant2t$, so the second moment of the third term is at most $2t\log^2(2k)$.
Combining the three bounds, we obtain the lemma.
\end{proof}

We now combine Lemma~\ref{lem:continuous} with Lemma~\ref{lem:spectrum} to prove Theorem~\ref{thm:window} (cf.~\cite[Lemmas~7 and~8]{S24}).

\begin{proof}[Proof of Theorem~\ref{thm:window}]
By the two conditions on $\omega$, we have $\sqrt{t_\omega}\,\log(2k)=o(t_\omega)$, so $\log^2(2k)=o(t_\omega)$, and both $t_\omega$ and $\omega$ tend to infinity.
By Lemma~\ref{lem:omega-bounds}, we have $t_\omega\leqslant\lceil\log n\rceil\leqslant k/8192$, so $\log^2(2k)=o(\log n)$, and hence \eqref{eq:regime} holds.
By Lemma~\ref{lem:continuous} with $t=t_\omega-1$ and with $t=t_\omega$, for large $n$ we have
\begin{equation}\label{eq:concentration}
\E\bigl[\bigl(-\log\mu_{Z,t}(X_t)-h_t\bigr)^2\bigr]\leqslant Ct_\omega\log^2(2k)=o(\omega^2).
\end{equation}

For the lower bound, let $t=t_\omega-1$.
By the minimality of $t_\omega$, we have $h_t<\log n-2\omega$.
For every $z$, fewer than $ne^{-\omega}$ vertices $x$ satisfy $\mu_{z,t}(x)>e^{\omega}/n$, and $d_{z}(t)$ is at least the sum of $\mu_{z,t}(x)-1/n$ over these vertices.
Hence
\[
\E\,d_{Z}(t)\geqslant\Pp\bigl(\mu_{Z,t}(X_t)>e^{\omega}/n\bigr)-e^{-\omega},
\]
and the probability on the right tends to one by \eqref{eq:concentration} and Chebyshev's inequality, since $h_t<\log n-2\omega$.

For the upper bound, we have $h_{t_\omega}\geqslant\log n-2\omega$, and hence, by \eqref{eq:concentration} and Chebyshev's inequality,
\begin{equation}\label{eq:density-bound}
\Pp\bigl(\mu_{Z,t_\omega}(X_{t_\omega})>e^{3\omega}/n\bigr)\to0.
\end{equation}
Fix $z\in G^k$, and define a function $f$ on $G$ by
\[
f(x)=\begin{cases}n\mu_{z,t_\omega}(x)&\text{if }n\mu_{z,t_\omega}(x)\leqslant e^{3\omega},\\0&\text{otherwise.}\end{cases}
\]
The measure with density $f$ with respect to $\pi$ is the part of $\mu_{z,t_\omega}$ carried by the vertices $x$ with $n\mu_{z,t_\omega}(x)\leqslant e^{3\omega}$, and the remaining part has total mass $1-\pi(f)$.
For every integer $j\geqslant0$, the matrix $P_{z}^j$ maps the first part to the measure with density $(P_{z}^{\top})^jf$, and hence
\[
\bigl\|\mu_{z,t_\omega}P_{z}^j-\pi\bigr\|_{\TV}\leqslant\frac12\bigl\|(P_{z}^{\top})^j\bigl(f-\pi(f)\bigr)\bigr\|_{2,\pi}+1-\pi(f)\leqslant\frac12\lambda_*(z)^je^{3\omega/2}+1-\pi(f),
\]
since $\|f-\pi(f)\|_{2,\pi}^2\leqslant\pi(f^2)\leqslant e^{3\omega}$.

Now suppose that $\lambda_*(z)^2\leqslant e^3\lceil\log n\rceil/k$.
Since $\lceil\log n\rceil\leqslant2\log n$ and $\log(k/\log n)\geqslant14\log2>9$, we have
\[
\log\frac1{\lambda_*(z)}\geqslant\frac12\log\frac k{\log n}-\frac{3+\log2}2\geqslant\frac14\log\frac k{\log n}.
\]
Let $j=\lceil8\omega/\log(k/\log n)\rceil$, so that $\lambda_*(z)^j\leqslant e^{-2\omega}$.
In discrete time, it follows that $d_{z}(t_\omega+j)\leqslant\frac12e^{-\omega/2}+1-\pi(f)$.
In continuous time, we have $\mu_{z,t_\omega+2j}=\sum_{m\geqslant0}\Pp(N=m)\,\mu_{z,t_\omega}P_{z}^m$, where $N$ is a Poisson random variable with mean $2j$, and therefore, by Chebyshev's inequality,
\[
d_{z}(t_\omega+2j)\leqslant\Pp(N<j)+\frac12e^{-\omega/2}+1-\pi(f)\leqslant\frac2j+\frac12e^{-\omega/2}+1-\pi(f).
\]
Since $\omega/\log(2k)\to\infty$, we have $j\to\infty$ and $2j\leqslant20\omega/\log(k/\log n)$ for large $n$.
As $d_{z}$ is non-increasing and $d_{z}\leqslant1$, in both cases we obtain
\[
\E\,d_{Z}\Bigl(t_\omega+\frac{20\,\omega}{\log(k/\log n)}\Bigr)\leqslant\Pp\Bigl(\lambda_*(Z)^2>\frac{e^3\lceil\log n\rceil}k\Bigr)+\frac2j+\frac12e^{-\omega/2}+\E\bigl[1-\pi(f)\bigr],
\]
where $f$ is defined with $Z$ in place of $z$.
The quantity $1-\pi(f)$ is the probability that the walk at time $t_\omega$ is at a vertex $x$ with $n\mu_{z,t_\omega}(x)>e^{3\omega}$, and its expectation over $Z$ tends to zero by \eqref{eq:density-bound}.
By Lemma~\ref{lem:spectrum}, the first term on the right is at most $1/n$, and the upper bound follows.
\end{proof}

\begin{proof}[Proof of Theorem~\ref{thm:main}]
By Lemma~\ref{lem:omega-bounds}, $t_*$ is finite for large $n$.
On the part of the sequence on which $\log k>(\log\log n)^2$, cutoff with high probability at time $t_*$ holds by Lemma~\ref{lem:large-range}, so it suffices to consider the part on which $\log k\leqslant(\log\log n)^2$.
On this part, we apply Theorem~\ref{thm:window} with $\omega=(\log n)^{2/3}$, so that $t_\omega=t_*$.
By Lemma~\ref{lem:omega-bounds}, we have $(\log n-2(\log n)^{2/3})/\log(2k)\leqslant t_*\leqslant\lceil\log n\rceil$, and since $\log(2k)\leqslant2(\log\log n)^2$ for large $n$, it follows that $\sqrt{t_*}\,\log(2k)=o\bigl((\log n)^{2/3}\bigr)$ and $(\log n)^{2/3}=o(t_*)$, which are the conditions of Theorem~\ref{thm:window}.
Hence $\E\,d_{Z}(t_*-1)\to1$ and $\E\,d_{Z}\bigl(t_*+20(\log n)^{2/3}/\log(k/\log n)\bigr)\to0$.
Fix $\delta\in(0,1)$.
For large $n$ in this part, we have $(1-\delta)t_*\leqslant t_*-1$ and $(1+\delta)t_*\geqslant t_*+20(\log n)^{2/3}/\log(k/\log n)$.
Since $d_{z}$ is non-increasing, it follows that
\[
\E\bigl[1-d_{Z}\bigl((1-\delta)t_*\bigr)\bigr]\to0\qquad\text{and}\qquad\E\,d_{Z}\bigl((1+\delta)t_*\bigr)\to0.
\]
By Markov's inequality, $d_{Z}\bigl((1-\delta)t_*\bigr)\inP1$ and $d_{Z}\bigl((1+\delta)t_*\bigr)\inP0$ on this part of the sequence, which completes the proof.
\end{proof}

\section*{Use of generative AI}

The author used ChatGPT 5.6 Plus to suggest candidate approaches, to assist with some of the proofs, to identify relevant terminology, and to proofread the manuscript. The mathematical content, the final text, and any remaining errors are
the responsibility of the author.

\end{document}